\documentclass[12pt]{article} % use larger type; default would be 10pt

\usepackage[utf8]{inputenc} % set input encoding (not needed with XeLaTeX)
\usepackage{snoj}
\usepackage{mathtools}
\newtheorem{theorem}{Theorem}
\newtheorem{definition}{Definition}
\newtheorem{lemma}{Lemma}
\newtheorem{example}{Example}
\newtheorem{proposition}{Proposition}
\newtheorem{remark}{Remark}
\newcommand{\Gr}{\operatorname{Gr}}

\newcommand{\copmon}{\langle X\rangle *[Y]}
\newcommand{\copalg}{k\langle X\rangle *k[Y]}
\newcommand{\supp}{\operatorname{supp}}
\usepackage{setspace}
\usepackage{biblatex} %Imports biblatex package
\usepackage{geometry} % to change the page dimensions
\usepackage{graphicx} % support the \includegraphics command and options

\usepackage{booktabs} % for much better looking tables
\usepackage{array} % for better arrays (eg matrices) in maths
\usepackage{paralist} % very flexible & customisable lists (eg. enumerate/itemize, etc.)
\usepackage{verbatim} % adds environment for commenting out blocks of text & for better verbatim
\usepackage{subfig} % make it possible to include more than one captioned figure/table in a single float
\usepackage{fancyhdr} % This should be set AFTER setting up the page geometry
\usepackage{sectsty}
\usepackage[nottoc,notlof,notlot]{tocbibind} % Put the bibliography in the ToC

\title{Brief Article}
\author{The Author}
\begin{document}
\begin{center}
{\Large \textbf{Commutativity of centralizers in a coproduct of a free algebra and a polynomial algebra}}

\bigskip
{\large Jakob Jurij Snoj}\\
{\href{mailto: jakob.snoj@fmf.uni-lj.si}{jakob.snoj@fmf.uni-lj.si}}\\
{\emph{Faculty of Mathematics and Physics (FMF), University of Ljubljana}}
\end{center}

\begin{abstract}
We show that the centralizer of a nonscalar element in the coproduct $\copalg$ of a free associative algebra and a polynomial algebra over a given field is commutative. For $k\langle X \rangle$ this is part of Bergman's centralizer theorem. Our proof relies on a reduction given in Bergman's proof and is of combinatorial nature, employing a strict order structure of the coproduct monoid.
\bigskip

\noindent\emph{Math.~Subj.~Class.:  16S10}
\end{abstract}

\section{Introduction}

Bergman's centralizer theorem states that the centralizer of a nonscalar element in a free associative algebra over a field $k$ is isomorphic to a univariate polynomial ring over $k$. The theorem was originally proven by G.~M.~Bergman in \cite{bergman} using the theory of free ideal rings (firs) and the $n$-term weak algorithm. A proof largely following the same ideas but differing in a few key steps is given by P.~M.~Cohn in \cite{cohn}. Later, the theorem was reproven by A.~Belov-Kanel, F.~Razavinia and W.~Zhang using generic matrices and quantization in \cite{quanti} and \cite{quanti2}. In \cite[Chapter~3]{sela}, Z.~Sela utilises a combinatorial approach to prove certain special cases.

We investigate the properties of the centralizer of a nonscalar element in a similar structure, differing from the free algebra roughly by requiring that some of the variables commute with each other, but not with the remaining variables. Formally, we consider the coproduct $\copalg$ of a free algebra and a polynomial algebra over a field $k$.

This construction recently appeared in connection with the rooted generalized hierarchical product of graphs, introduced in \cite{gproduct} by L.~Barrière, C.~Dalfó, M.~A.~Fiol and M.~Mitjana and later discussed by W.~Imrich, R.~Kalinowski, M.~Pilśniak in \cite{graphs1} and W.~Imrich, I.~Klep and D.~Smertnig in \cite{graphs}.

It is also a special case of a monoid algebra over a graph monoid. Given a graph, the corresponding graph monoid --- also called a trace monoid or a free partially commutative monoid --- is obtained by considering the free monoid generated by the vertex set and quotienting it by the congruence generated by the relations $ab=ba$ for every edge $\{a, b\}$. This construction was introduced by P.~Cartier and D.~Foata in \cite{gp1} and is itself a special case of a graph product of monoids, studied for instance by D.~Yang and V.~Gould in \cite{gp2}. It is worth mentioning that the term 'graph product' is used here to denote a notion completely unrelated to the construction discussed in the previous paragraph.

We prove the following partial generalization of Bergman's centralizer theorem, analogous to a main step in Bergman's original proof of this result. 

\begin{theorem}\label{glavni}
Let $X$ and $Y$ be sets and $k$ a field. Let $u$ be a nonscalar element of $k\langle X\rangle * k[Y]$. Then the centralizer of $u$ in $\copalg$ is commutative. 
\end{theorem}

In Bergman's original paper, the proof of the corresponding property is reduced by \cite[Proposition~2.1]{bergman} to showing that the centralizer's corresponding graded ring is at most one-dimensional in each degree as an algebra over $k$. The same reduction serves as the starting point for our proof. Bergman's remaining argument is short but relies heavily on free algebras satisfying a property called the 2-term weak algorithm, which is a property not enjoyed by our coproduct, so a different approach must be taken in this case.

The approach taken in this paper makes heavy use of the fact that the monoid consisting of the basis elements of the given coproduct algebra can be strictly ordered, i.e., ordered in a way that respects the product operation (\cite{total2} and \cite{total1}). The requirement that the components of the centralizer's corresponding graded ring are at most one-dimensional can be rephrased in terms of homogeneous elements, and by some reductions, it suffices to consider the support of a product of two homogeneous commuting elements of equal total degree. While this support could be easily described in terms of the supports of individual factors in a free algebra, this is considerably less straightforward in the coproduct. Using the well-order of basis elements allows us to give a partial description of this support, and this turns out to be sufficient to complete the proof.

\section{Preliminaries}
We work in the coproduct $k\langle X\rangle * k[Y]$ of a free algebra $k\langle X \rangle$ and a polynomial algebra $k[Y]$ over a field $k$ and sets of variables $X$ and $Y$, which can also be infinite or empty. We refer to the elements of sets $X$ and $Y$ as \emph{noncommuting} and \emph{commuting variables}, respectively.

This coproduct algebra is the monoid $k$-algebra of the monoid $\langle X\rangle * [Y]$, which is the coproduct of the free monoid $\langle X\rangle$ and the free abelian monoid $[Y]$. Thus, elements of $\langle X\rangle * [Y]$ are products of elements of the disjoint union $X\sqcup Y$, with two such products equal if they differ only in the order of consecutive variables in $Y$.

For any element $w=a_1a_2\dots a_n\in \copmon$ with each $a_i\in X\sqcup Y$, its \emph{length} is $|w|=n$. The \emph{support} $\supp x$ of an element $x\in \copalg$ is the set of all elements of $\copmon$ appearing in $x$ with nonzero coefficient.

Recall that a \emph{(positive) filtration} of an algebra $R$ is a family of vector subspaces $$R_0\subseteq R_1\subseteq R_2 \subseteq\ldots$$ such that 
\begin{itemize}
\item $R = \bigcup_{i\geq 0} R_i$,
\item $R_iR_j \subseteq R_{i+j}$ for all $i$,~$j \geq 0$ and $1 \in R_0$.
\end{itemize}
Given a filtered algebra $R$, the \emph{degree} of a nonzero element $a$ is $d(a)=\min\{\,i\mid a\in R_i\,\} \in \Z_{\geq 0}$. The degree of $0$ is defined to be $-\infty$.

A filtered algebra $R$ also has an associated graded algebra $\Gr R$: 
as vector subspaces,
\[
\Gr R \coloneqq \bigoplus_{i \geq \Z} (\Gr R)_i \qquad\text{with}\qquad{(\Gr R)_i \coloneqq R_i / R_{i-1}},
\]
and the multiplication map is induced from $R$, extended linearly from $(\Gr R)_i \times (\Gr R)_j \to (\Gr R)_{i+j}$ given by $(a, b) \mapsto ab$. 
The \emph{leading term} of $0 \ne x\in R$ is the image of $x$ under the canonical quotient map $R_{d(x)}\rightarrow (\Gr R)_{d(x)}$ and for each integer $n$, the subspace $(\Gr R)_n$ is said to consist of \emph{homogeneous elements of degree $n$.}

The algebra $\copalg$ is filtered by total degree and canonically isomorphic to its associated graded algebra.
For $x\in\copalg$, we denote its leading term by $\overline{x}$. We say an element $x\in\copalg$ is \emph{homogeneous of degree $n$} if $\supp x$ only consists of elements of length $n$. The homogeneous elements of a fixed degree $n$ of $\copalg$ form a vector space canonically isomorphic to $(\Gr(\copalg))_n$.

An important ingredient of the proof is the following proposition of G.~M.~Bergman, used in \cite[Proposition~2.1]{bergman} as a part of his proof of Bergman's centralizer theorem.

\begin{proposition} \label{bergman}
Let $C$ be a $\Z$-filtered algebra over a field $k$ in which the relation $d(ab)=d(a)+d(b)$ is satisfied for all $a, b\in C$. If for all $n$, the $k$-vector space $(\Gr C)_n$ is $0$ or $1$-dimensional, then $C$ is commutative.
\end{proposition}

Bergman applies this result to a centralizer of a nonscalar element in a free algebra over a field.
The verification of the condition of the proposition depends on the free algebra satisfying the $2$-term weak algorithm (for a definition, see \cite[Section~2.4]{cohn} or \cite[Definition~1.3]{bergman}). As the following straightforward example demonstrates, this property does not hold for $\copalg$.
\begin{example}
Let $Y=\{y_1, y_2\}$. We consider $k[Y]$ as a subalgebra of $\copalg$. The elements $y_1$ and $y_2$ have a common right multiple $y_1y_2$, but the right ideal generated by $y_1$ and $y_2$ is clearly not principal.
\end{example}

In our proof, we thus take a different approach. Our main arguments are of combinatorial nature and rely heavily on a total order of the coproduct monoid.

\section{A strict total order and some non-vanishing properties}

We first introduce a type of order compatible with the monoid structure.

\begin{definition}
A monoid $(M, \cdot, 1)$ is \emph{strictly ordered} if there exists a total order $\leq$ on $M$ such that for all $a$, $b$, $c\in M$, $$a< b\text{ implies } ac<bc \text{ and } ca<cb.$$
\end{definition}

The monoid $\copmon$ can be strictly ordered as follows. We order each of the sets $X$ and $Y$ with any total order. We then order the elements of the free monoid $\langle X\rangle$ with the \emph{shortlex} order: for two elements $a, b\in \langle X\rangle$, we set $a<b$ if either $|a|<|b|$, or $|a|=|b|$ and the leftmost variable in which $a$ and $b$ differ is smaller in $a$ than in $b$.

On $[Y]$, we use the lexicographical order: we set $1<a$ for all nonempty words $a\in[Y]$ and, factorizing each element uniquely into a product of the form $y_1^{\alpha_1}\cdots y_r^{\alpha_r}$, where the $y_i$ are variables satisfying $y_1\leq\cdots \leq y_r$ and $\alpha_i> 0$, we set $a<b$ if the smallest factor in which $a$ and $b$ differ either appears only in $b$ or appears in $b$ with bigger exponent than in $a$.

Since $\langle X\rangle$ and $[Y]$ are strictly ordered, the coproduct $\langle X\rangle *[Y]$ can also be ordered with a strict order (this fact is not trivial and appears in \cite{total2} and \cite[Theorem~16]{total1}).
We will henceforth denote this total order with the symbol $<$.

In a free algebra, the support of a product of two homogeneous elements $a$ and $b$ is uniquely determined by the supports of the individual elements: it consists of exactly the elements $xy$ with $x\in \supp a$ and $y\in \supp b$. In the algebra $\copalg$, this is not necessarily the case, as demonstrated by the following straightforward example.

\begin{example}
Let $X=\{x\}$ and $Y=\{y_1, y_2\}$. In $\copalg$, we have
$$(xy_1+xy_2)(y_1x-y_2x)=xy_1^2x-xy_2^2x.$$
Indeed, as $y_1$ and $y_2$ commute, the term $xy_1y_2x$ appears twice with opposite coefficients, so this term vanishes in the product.
\end{example}

One of the ideas of our proof is to try to control this phenomenon and describe some cases when this vanishing cannot occur. We will see how to do this with the help of our total order. First, we isolate the relevant parts of the elements of $\copmon$.

\begin{definition}
    Let $w\in \copmon$. The \emph{prefix} $p(w)$ is the product of all variables of $Y$ appearing before the first variable of $X$ in $w$. Symmetrically, the \emph{suffix} $s(w)$ is the product of all variables of $Y$ appearing after the last variable of $X$ in $w$.
\end{definition}

If $w$ consists only of commuting variables, we set $p(w)=s(w)=w$. We will sometimes treat such elements separately.

\begin{definition}
    Let $f = \sum_{i\ge 0} f_i \in \copalg$ be nonzero with homogeneous components $f_i$.
    \begin{enumerate}
    \item The element $f$ is \emph{pure} if $f \in k[Y]$, and \emph{non-pure} otherwise.
    \item The element $f$ is $m$-pure for $m \in \Z_{\ge 1}$ if $f_m$ is pure.
    \item The element $f$ is top-pure if it is $d$-pure, where $d$ is the total degree of $f$.
    \end{enumerate}
\end{definition}

Next, we introduce an equivalence relation on $\copmon$.

\begin{definition}
We write $u\sim v$ if $|u|=|v|$ and either $u$ and $v$ are both pure, or, writing $u=p(u)u_1s(u)$ and $v=p(v)v_1s(v)$, we have $|p(u)|=|p(v)|$, $u_1=v_1$ and $|s(u)|=|s(v)|$. We denote the equivalence class of $x$ under this relation by $[x]$.
\end{definition}
In other words, each equivalence class of $\sim$ consists of all elements of $\copmon$ of the same length with the first and last noncommuting variable (if any) occuring respectively in the same positions, and all variables between them forming the same element.

\begin{lemma}\label{research1}
Let $X$ and $Y$ be sets. Let $v$, $w$, $v'$ and $w'$ be elements of $\copmon$ satisfying the relation $vw=v'w'$ such that $|v|=|v'|$ and $|w|=|w'|$. Then, the following holds.
\begin{enumerate}[(i)]
\item $v\sim v'$ and $w\sim w'$.
\item If $v$ is non-pure, then $p(v)=p(v')$.
\item If $w$ is non-pure, then $s(w)=s(w')$.
\item $s(v)p(w)=s(v')p(w')$.
\end{enumerate}
\end{lemma}

\begin{proof}
\emph{(i)}\;\; Let $n=|v|=|v'|$, and $m=|w|=|w'|$. We prove $v\sim v'$, the other relation then follows analogously.

If $v$ is pure, the first $n$ variables of $v'w'=vw$ are commuting variables, so $v'$ is also pure and $v\sim v'$ follows immediately. In the remaining case, we can write $vw=p(v)v_1s(v)w=v'w'$. Since the element $v_1$ begins and ends with a non-commuting variable, it must appear in the same position in $v'w'$ and hence $v'$, similarly, all variables in the other positions of $v'$ are commuting by construction. Consequently, the element $v'$ can only differ from $v$ in its suffix, so $v\sim v'$ by definition.

\emph{(ii), (iii)}\;\; The relation $vw=p(v)v_1s(v)w=v'w'$ also directly implies that $p(v)$ is the prefix of $v'$, proving the second item. The third is then proven symmetrically.

\emph{(iv)}\;\; Finally, if $v$ and $w$ are non-pure, we can write $$vw=p(v)v_1s(v)p(w)w_1s(w)=p(v')v_2s(v')p(w')w_2s(w')=v'w'.$$ By the preceding arguments, we have $p(v)=p(v')$, $v_1=v_2$, $s(w)=s(w')$ and $w_1=w_2$. Since $\copmon$ is a cancellation monoid, it follows that their middle parts must also be equal. If $v$ or $w$ is pure, we can adjust the above equality accordingly and conclude in the same way.
\end{proof}

We will often also be interested in a subset of the support of an element. For an element $a\in \copalg$ and an element $w\in \copmon$, we will use the notation $\supp_w a$ to denote the intersection $\supp a\cap [w]$.

The following lemma now allows us to describe some elements of the support of a product of two homogeneous elements.

\begin{lemma}\label{researchlemma}
Let $X$ and $Y$ be sets. Let $a$ and $b$ be two homogeneous elements of $\copalg$. Let $x\in \supp a$ and $y\in \supp b$ be elements such that the following two conditions hold:
\begin{enumerate}
\item The suffix $s(x)$ is maximal in $\{s(w)\mid w\in \supp_x a\}$, or $x$ is non-pure and the suffix $s(x)$ is maximal in $\{s(w)\mid w\in\supp_x a,\;p(w)=p(x)\}.$
\item The prefix $p(y)$ is maximal in $\{p(w)\mid w\in \supp_y b\}$, or $y$ is non-pure and the prefix $p(y)$ is maximal in $\{p(w)\mid w\in\supp_y b,\; s(w)=s(y)\}.$
\end{enumerate}
Then, we have $xy\in \supp ab$.
\end{lemma}

\begin{proof}
It will suffice to prove that $xy$ cannot arise as a product of two other elements in $\supp a$ and $\supp b$, respectively. Let $xy=x'y'$ for some $x'\in\supp a$ and $y'\in\supp b$. By recalling $a$ and $b$ are homogeneous and using Lemma \ref{research1}, we observe $x'\in \supp_x a$ and $y'\in\supp_y b$. If $x$ is non-pure, we also obtain $p(x)=p(x')$, and if $y$ is non-pure, we obtain $s(y)=s(y')$.

Whether $x$ is pure or non-pure, the assumption now allows us to argue by maximality to obtain $s(x)\geq s(x')$. Symmetrically, we obtain $p(y)\geq p(y')$. By recalling $\copmon$ is a strictly ordered monoid and using Lemma \ref{research1} again, we obtain $s(x)p(y)=s(x')p(y')$, so both inequalities must be equalities. In particular, this shows $x=x'$ and $y=y'$, leading to the final contradiction.
\end{proof}

\section{Considering commuting homogeneous elements}

We next determine how pairs of commuting homogeneous elements of equal total degree are related.

\begin{lemma} \label{research2}
Let $X$ and $Y$ be sets and $k$ a field. Let $a$, $b\in k\langle X\rangle * k[Y]$ be two commuting homogeneous elements of equal total degree. Suppose $v_0\in \supp a$ is non-pure.  Then, the sets $\supp_{v_0} a$ and $\supp_{v_0} b$ are nonempty. Additionally, we have the equalities
\begin{align*}\{\,p(v)\,\mid\,v\in\supp_{v_0} a\,\}&=\{\,p(w)\,\mid\,w\in\supp_{v_0} b\,\},\quad\text{and}\\
\{\,s(v)\,\mid\,v\in\supp_{v_0} a\,\}&=\{\,s(w)\,\mid\,w\in\supp_{v_0} b\,\}.
\end{align*}
\end{lemma}

\begin{proof}
%The set $\mathcal{A}_{v_0}$ is nonempty by definition. Without loss of generality, we can assume $v_0$ is chosen such that $s(v_0)$ is maximal in $\mathcal{A}_{v_0}$. Let $w\in \supp b$ be such that $p(w)$ is maximal in $\mathcal{B}_w$. By Proposition \ref{researchlemma}, we have $v_0w\in \supp ab=\supp ba$, implying $v_0\in\mathcal{B}_{v_0}$ by Proposition \ref{research1}.

We notice $\supp_{v_0} a$ is nonempty by definition, so it suffices to prove $\{\,p(v)\,\mid\,v\in\supp_{v_0} a\,\}\subseteq\{\,p(w)\,\mid\,w\in\supp_{v_0} b\,\}$. Indeed, nonemptiness of $\supp_{v_0} b$ will then follow immediately, while the converse inclusion and remaining equality will follow by symmetry.

Let $v\in\supp_{v_0} a$ be such that $s(v)$ is maximal in $\{s(u)\mid u\in\supp_{v_0} a,\; p(u)=p(v)\}.$ We fix an element $w\in\supp b$ with $p(w)$ maximal in $\supp_w b$. By Lemma \ref{researchlemma}, we obtain $vw\in \supp ab= \supp ba$, so $vw=w'v'$ with $w'\in\supp b$ and $v'\in\supp a$. By Lemma \ref{research1}, $w\in\supp_{v_0} b$ and $p(w')=p(v)$, proving the claim.
\end{proof}

To apply Proposition \ref{bergman}, it suffices to consider homogeneous elements, as we work within a single component of $\Gr(\copalg)$. We will first consider the case of two homogeneous commuting elements of equal degree and later turn to higher generality.

The main step in our proof consists of two largely similar lemmas, each treating a single case depending on the purity of the commuting elements.

\begin{lemma}\label{mainone}
Let $X$ and $Y$ be sets and $k$ a field. Let $a$ and $b$ be two commuting elements of equal total degree $n$ in $R=k\langle X\rangle * k[Y]$. If at least one of $a$ or $b$ is not top-pure, there exists a scalar $\lambda\in k$ such that $\overline{a}=\overline{\lambda b}$ in $(\Gr R)_n$.
\end{lemma}

\begin{proof}
As we are working within a single component of $\Gr R$, we can equivalently assume $a$ and $b$ are homogeneous elements of equal total degree. Without loss of generality, let $b$ be non-pure. We wish to prove there exists a scalar $\lambda\in k$ such that $a=\lambda b$.

Since $b$ is non-pure, there exists a non-pure element $w\in \supp b$. By Lemma $\ref{research2}$, the sets $\supp_w a$ and $\supp_w b$ are both nonempty and the sets of prefixes (and suffixes) of their elements are equal.

Let $\supp_w^p a$ and $\supp_w^p b$ be the subsets of $\supp_w a$ and $\supp_w b$, respectively, consisting of the elements with the maximal prefix. Let $u\in \supp_w^p a$ and $v\in \supp_w^p b$ be such that $s(u)$ and $s(v)$ are as large as possible among the suffixes of elements in the respective sets. By symmetry, we may also assume $s(u)\leq s(v)$.

Since the prefix of $u$ is maximal by definition, Lemma \ref{researchlemma} implies $vu\in \supp ba= \supp ab$. We can thus write $vu=u_1v_1$, where $u_1\in \supp_w a$ and $v_1\in \supp_w b$ by Lemma \ref{research1}. Using this lemma, we can make some further observations:

Since $u_1$ is non-pure, we have $p(u_1)=p(v)$. In particular, this implies $u_1\in\supp_w^p a$. Next, since $u_1\in\supp_w^p a$, it follows $s(u_1)\leq s(u)\leq s(v)$ by our choice of $u$.
Lastly, by our choice of $v$, we have $p(v_1)\leq p(v)=p(u)$.

To combine the above, we apply Lemma \ref{research1} one final time to observe that $vu=u_1v_1$ implies $s(v)p(u)=s(u_1)p(v_1)$. Since we have $s(u_1)\leq s(v)$ and $p(v_1)\leq p(u)$, this can only hold if both inequalities are equalities. In particular, this shows $u_1=v$, so we obtain $v\in \supp a$.

To conclude, let $\lambda\in k$ be such that $v\not\in \supp (a-\lambda b)$. Since the element $a-\lambda b$ also commutes with $b$, this is only possible if it is trivial, implying $a=\lambda b$.
\end{proof}

In case $a$ and $b$ are both top-pure, the key arguments are similar, but some additional care must be taken. We begin with the following two lemmas.

\begin{lemma}\label{research3}
Let $X$ and $Y$ be sets and $k$ a field. Let $a$ and $b$ be two commuting elements of $k\langle X\rangle * k[Y]$ of equal total degree. If $a$ and $b$ are top-pure and $b$ is non-pure, then $a$ is also non-pure. Additionally, there exists a positive integer $\ell$ such that neither $a$ nor $b$ are $\ell$-pure but both are $m$-pure for all $m>\ell$.
\end{lemma}

\begin{proof}
Let $n$ be the total degree of $a$ and $b$, and let $\ell$ be the largest positive integer such that $b$ is not $\ell$-pure. We will show that $a$ is not $\ell$-pure but is $m$-pure for all $m\geq \ell$. We assume the contrary. By switching the roles of $a$ and $b$ and increasing $\ell$ appropriately if required, we may assume that $a$ is $m$-pure for all $m\geq \ell$.

Let $w \in \supp b$ be an non-pure element of length $\ell$. We can choose $w$ such that $p(w)$ is maximal in  $\{p(u)\mid u\in\supp_w b\}$. Taking $v\in \supp a$ to be the maximal element of length $n$, Lemma \ref{researchlemma} implies $vw\in \supp ab= \supp ba$.

Since all the elements of $\supp a$ and $\supp b$ of length larger than $\ell$ are pure and elements of $\supp a$ of length $\ell$ are also pure, we must have $vw=v'w'$ with $v'\in \supp b$, $|v'|=\ell$, and $w'\in\supp a$, $|w'|=n$. But then, we have $|s(vw)|=|s(v'w')|\geq n>|w|$, so $w$ must be pure, a contradiction.
\end{proof}

\begin{lemma}\label{aux}
Let $X$ and $Y$ be sets and $k$ a field. Let $a$ and $b$ be two commuting elements of equal total degree $n$ in $R=k\langle X\rangle * k[Y]$. If $a$ and $b$ are top-pure, then for every non-pure element $w \in\supp b$ such that all elements $v\in\supp b$ with $|v|>|w|$ are pure, we have
$$\{\,p(u)\mid u\in\supp_w a\,\}=\{\,p(u)\mid u\in\supp_w b\,\}.$$
\end{lemma}

\begin{proof}
By Lemma \ref{research3}, there exists a positive integer $\ell$ such that neither $a$ nor $b$ are $\ell$-pure but both are $m$-pure for all $m>\ell$.  Let $w$ be a non-pure length $\ell$ element in $\supp b$. We consider an element $v\in\supp_w b$ such that $s(v)$ is maximal in $\{s(u)\mid u\in \supp_w b,\; p(u)=p(v)\}$. It suffices to prove that there exists an element $v_1\in \supp_w a$ with $p(v_1)=p(v)$. Indeed, this will immediately imply $\{p(u)\mid u\in\supp_w a\}\supseteq\{p(u)\mid u\in\supp_w b\},$ and the converse direction follows by symmetry.

Let $z$ be a non-pure element of length $\ell$ in $\supp a$ such that $p(z)$ is maximal in $\{p(u)\mid u\in \supp_z a\}$. In particular, since $z$ and $v$ are non-pure, $|p(vz)|$ and $|s(vz)|$ are both smaller than $\ell$, so if $vz=v'z'$ for some $v'\in \supp b$, $z'\in\supp a$, we must have $|v'|=|z'|=\ell$, as all elements of length larger than $\ell$ of both $\supp a$ and $\supp b$ are pure.

By considering only the homogeneous degree $\ell$ parts of $a$ and $b$, Lemma \ref{researchlemma} along with the preceding paragraph implies that $vz\in\supp ba=\supp ab$. We thus have $vz=v_1z_1$ with $v_1\in \supp a$ and $z_1\in\supp b$ and $|v_1|=|z_1|=\ell$, so Lemma \ref{research1} implies that $v_1\in\supp_w a$ and $p(v_1)=p(v)$.
\end{proof}

\begin{lemma}\label{maintwo}
Let $X$ and $Y$ be sets and $k$ a field. Let $a$ and $b$ be two commuting elements of equal total degree $n$ in $R=k\langle X\rangle * k[Y]$. If $a$ and $b$ are top-pure but $b$ is non-pure, there exists a scalar $\lambda\in k$ such that $\overline{a}=\overline{\lambda b}$ in $(\Gr R)_n$.
\end{lemma}

\begin{proof}Let $v$ be the maximal length $n$ element of $\supp b$. Since $|v|=n$, the element $v$ is pure and we have the equality $v=p(v)=s(v)$. By symmetry, we may assume $v\geq u$ for all $u\in\supp a$, $|u|=n$.

Let $u_1$ be an element in $\supp a$ such that $|u_1|=\ell$ and $p(u_1)$ is maximal in $\{p(w)\mid w\in \supp_{u_1} a,\;p(w)=p(u_1)\}$. Since basis elements of length larger than $\ell$ are pure, it follows that if $vu_1=v'u_1'$ holds for some $v'\in \supp b$ and $u_1'\in \supp a$, we necessarily have $|v'|=n$ and $|u_1'|=\ell$, but the maximality of $v$ and $p(u_1)$ then forces $v=v'$ and $u_1=u_1'$, so $vu_1\in \supp ba=\supp ab$.

%By applying Lemma \ref{researchlemma} on the homogeneous parts of $a$ and $b$ of degrees $\ell$ and $n$, respectively, we obtain 
We thus have $vu_1=v'u_1'$ with $v'\in\supp a$, $|v'|=n$, and $u_1'\in\supp b$, $|u_1'|=\ell$. By Lemma \ref{research1}, we have $u_1'\in \supp_{u_1} b$ and by Lemma \ref{aux}, the prefix $p(u_1)$ is also maximal in the $\{p(w)\mid w\in \supp_{u_1} b,\;p(w)=p(u_1)\}$. Additionally, $|v'|=n$ implies the element $v'$ is pure, so we have $v'=p(v')=s(v')$.

Combining the above properties, we obtain $v'\leq v$ and $p(u_1')\leq p(u_1)$. This implies
$$p(vu_1)=p(v'u_1')=v'p(u_1')\leq vp(u_1)=p(vu_1),$$
so equality must hold throughout. In particular, this shows $v=v'$, so we can proceed just as in Lemma \ref{mainone}: it follows that $v\in \supp a$, so there exists a scalar $\lambda\in k$ with $v\not \in \supp (a-\lambda b)$. By repeating the above steps with $a-\lambda b$ in place of $a$ and recalling this element also commutes with $b$, we see that $v\in\supp (a-\lambda b)$ would lead to a contradiction, so the degree of $a-\lambda b$ is strictly smaller than $n$, concluding the proof.
\end{proof}

\section{Commutativity of the centralizer}

We are now prepared to prove our main theorem.

\begin{proof}[Proof of Theorem \ref{glavni}] Let $C=C(u)$ be the centralizer of $u$ in $R=\copalg$. If the element $u$ is pure, it clearly commutes with all pure elements of $R$. If $a$ is a non-pure element commuting with $u$, there exist integers $p$ and $q$ such that $a^p$ and $u^q$ have the same total degree. Now, since a power of a non-pure element is non-pure, the element $a^p$ is non-pure and commutes with $u^q$, contradicting Lemma \ref{mainone} or Lemma \ref{research3}. Consequently, the centralizer in this case is precisely $k[Y]$ and is thus commutative.

Next, we assume the element $u$ is non-pure. We note that $\Gr C$ inherits the graded algebra structure from the associated graded algebra $\Gr R$. We will show that each component of $\Gr C$ is at most $1$-dimensional over $k$, which will then imply the commutativity of $C$ by Proposition \ref{bergman}.

We first employ the strict order of the elements of $\langle X\rangle *[Y]$. For an element $x\in R$, we denote by $\varphi(x)$ the largest element of $\copmon$ with respect to this order appearing with nonzero coefficient in $\overline{x}$. We denote this coefficient by $c(x)$.

Let $a$ and $b$ be two elements of $C$ of equal total degree $n$. By scaling appropriately, we can assume $c(a)=c(b)=1$. As before, we note that there exist positive integers $p$ and $q$ such that $a^p$, $b^p$ and $u^q$ have equal total degree. Since the monoid ordering is strict, we have the relation
$$\varphi(a^p)=\varphi(a)^p\quad\text{and}\quad\varphi(b^p)=\varphi(b)^p.$$
As $\varphi(a^p)$ can only appear in $\overline{a^p}$ as a product of $p$ copies of $\varphi(a)$, its coefficient in $\overline{a^p}$ must equal $1$. Analogously, we obtain $c(b^p)=1$.

If either of $a^p$ and $u^q$ is not top-pure, we can now apply Lemma \ref{mainone}. Otherwise, we apply Lemma \ref{maintwo}. Analogously, we apply the appropriate proposition with $b^p$ in place of $a^p$. We thus deduce that there exist coefficients $\lambda_1, \lambda_2\in k$ satisfying
$$\overline{a^p}=\lambda_1\overline{u^q}\quad\text{and}\quad\overline{b^p}=\lambda_2\overline{u^q}.$$
From the above relation, we can now deduce
$$\lambda_1c(u^q)=c(\lambda_1u^q)=c(a^p)=c(b^p)=c(\lambda_2u^q)=\lambda_2c(u^q),$$
implying $\lambda_1=\lambda_2$. This shows the equality $\overline{a^p}=\overline{b^p}$, in particular implying $\varphi(a^p)=\varphi(b^p)$, from which we can conclude $\varphi(a)=\varphi(b)$.

It remains to prove the element $a-b$ has strictly smaller total degree than $n$, implying the leading terms of $a$ and $b$ are equal, as desired. Arguing by contradiction, we assume $a-b$ has total degree $n$. Since this element commutes with $u$, we can repeat the above arguments to conclude $\overline{a^p}=\overline{(a-b)^p}.$

We now compare $\varphi(a^p)$ and $\varphi((a-b)^p)$. On one hand, the relation $\overline{a^p}=\overline{(a-b)^p}$ shows $\varphi(a^p)=\varphi((a-b)^p)$. On the other hand, we notice $\varphi(a)=\varphi(b)$, so this term vanishes in $a-b$ and we have $\varphi(a-b)<\varphi(a)$. This implies $\varphi((a-b)^p)<\varphi(a^p)$, leading to a final contradiction.
\end{proof}

\begin{remark}
Once the $0$ or $1$-dimensionality of $(\Gr C)_n$ for all $n$ is proven, the argument given in the proof of \cite[Proposition~2.2]{bergman} can be applied verbatim to show that $C$ is a module-finite integral extension of $k[u]$.
\end{remark}

\paragraph{Acknowledgements} The author was supported by the Slovenian Research and Innovation Agency (ARIS) program P1-0288.
\printbibliography
\end{document}